\documentclass[11pt]{amsart}

\usepackage{amsmath,amssymb,amsfonts,amsthm,mathtools,mathrsfs,float,graphicx}
\usepackage[colorlinks=true,linkcolor=black,citecolor=blue,urlcolor=blue]{hyperref}

\newtheorem{theorem}{Theorem}[section]
\newtheorem{proposition}[theorem]{Proposition}
\newtheorem{lemma}[theorem]{Lemma}
\newtheorem{corollary}[theorem]{Corollary}
\newtheorem{conjecture}[theorem]{Conjecture}
\newtheorem{remark}[theorem]{Remark}

\newcommand{\Z}{\mathbb Z}
\newcommand{\Q}{\mathbb Q}
\newcommand{\curly}[1]{\left\{#1\right\}}
\newcommand{\Gauss}[2]{\begin{bmatrix}#1\\#2\end{bmatrix}_{Q}}

\title[Cyclotomic expansions via Bailey transforms]
{Cyclotomic expansions of colored $SU(n)$ invariants of two-strand torus knots and Bailey transforms}

\author{Chuwen Wang}

\begin{document}
\maketitle

\vspace{-1.8em}
\begin{center}
{\small Department of Mathematics, Renmin University of China}\\[-2pt]
{\small
\href{mailto:chuwenwang@ruc.edu.cn}
{\textcolor{blue}{chuwenwang@ruc.edu.cn}}
}
\end{center}
\vspace{1em}

\begin{abstract}
Habiro's cyclotomic expansion of the colored Jones polynomial has a
higher rank analogue conjectured by Chen--Liu--Zhu for colored
\(SU(n)\) invariants.  We prove this conjecture for torus
knots \(T(2,2p+1)\) and give a Bailey theoretic realization of the
cyclotomic coefficients.  For \(n\geq2\) and
\(K_p=T(2,2p+1)\), the colored $SU(n)$ invariants admit an
expansion
\[
J_N^{SU(n)}(K_p;q)
=
\sum_{m=0}^{N}
\left(\prod_{j=0}^{m-1}\{N-j\}\{N+n+j\}\right)
H_m^{(n,p)}(q),
\]
where \(H_m^{(n,p)}(q)\in\mathbb Z[q^{\pm1}]\) is independent of the
color \(N\).

The proof identifies the cyclotomic basis with a Newton basis and rewrites the resulting Newton coefficients as an ordinary Bailey transform. 
The Lin--Zheng formula for \(T(2,2p+1)\) then gives a well-poised Bailey kernel.
A terminating very-well-poised \({}_6\phi_5\) summation diagonalizes the kernel and reduces the integrality problem to an ordinary Bailey transition. 
We prove a uniform integrality theorem for these transitions in a formal integral \(q\)-difference operator algebra.

As a direct corollary, the expansion yields the corresponding congruence
relations and proves part~{\rm (i)} of the Chen--Liu--Zhu
\(SU(n)\) volume conjecture for \(T(2,2p+1)\).

\end{abstract}

\setcounter{tocdepth}{2}
{\hypersetup{linkcolor=blue}\tableofcontents}

\section{Introduction}
\label{sec:introduction}

Habiro's cyclotomic expansion of the colored Jones polynomial is a
fundamental integrality structure in quantum topology. 
The color dependence is carried by explicit cyclotomic factors, with
Laurent polynomial coefficients independent of the color; 
see \cite{HabiroSL2,HabiroCyclotomic,Habiro,HabiroLe}.
A natural question is whether an analogous cyclotomic structure persists
for colored \(SU(n)\) invariants.

Let \(J_N^{SU(n)}(K;q)\) denote the
colored \(SU(n)\) invariant of a zero-framed knot \(K\), associated with the
trivial partition \((N)\), and normalized by
\[
J_N^{SU(n)}(U;q)=1
\]
for the unknot \(U\).  

Motivated by congruence relations for colored HOMFLY-PT invariants \cite{CLPZ,ZhuStructures},
Chen--Liu--Zhu proposed the following \(SU(n)\) generalization of
Habiro's cyclotomic expansion.

\begin{conjecture}[Chen--Liu--Zhu]\label{conj:CLZ}
For every knot \(K\) and every \(n\geq2\), there exist Laurent polynomials
\[
H_m^{(n)}(K;q)\in\mathbb Z[q^{\pm1}],
\]
independent of the color \(N\), such that
\begin{equation}\label{eq:intro-CLZ}
J_N^{SU(n)}(K;q)
=
\sum_{m=0}^{N}
C_{N+1,m}^{(n)}(q)\,H_m^{(n)}(K;q),
\end{equation}
where
\begin{equation}\label{eq:intro-C}
C_{N+1,m}^{(n)}(q)
=
\prod_{j=0}^{m-1}
\{N-j\}\{N+n+j\}.
\end{equation}
\end{conjecture}

The conjecture was proved for the figure-eight knot and the trefoil knot in \cite{CLZ}.
Several related forms of cyclotomic expansion are known.  
For colored Jones polynomials, Masbaum gave a skein theoretic derivation of Habiro-type formulas \cite{Masbaum}. 
Hikami--Lovejoy obtained Bailey pair formulas for the two-strand torus knots \(T(2,2p+1)\) \cite{HikamiLovejoy}.
Related cyclotomic expansions were also studied for generalized Jones polynomials \cite{BerestGallagherSamuelson} and for colored superpolynomials \cite{ChenSuper}.
In higher rank, Kameyama--Nawata--Tao--Zhang formulated a cyclotomic expansion conjecture for HOMFLY-PT polynomials colored by rectangular Young diagrams \cite{KNTZ}.
Beliakova--Gorsky constructed a cyclotomic expansion for \(\mathfrak{gl}_N\) invariants using interpolation Macdonald polynomials and a completion of the center of \(U_q(\mathfrak{gl}_N)\) \cite{BeliakovaGorsky}.  
Their expansion differs from Habiro's expansion considered here and, even in the \(\mathfrak{sl}_2\) specialization, is not written in Habiro's standard cyclotomic basis.
Explicit HOMFLY-PT formulas and cyclotomic expansions for particular knot families were studied in \cite{CLZDoubleTwist,Kawagoe,KawagoeTwist}.
An earlier calculation directly relevant here is due to Wei: in an unpublished master's-thesis manuscript, the \(SU(3)\), \(N=2\) case of \eqref{eq:intro-CLZ} was proved for \(T(2,2p+1)\) by a divisibility induction \cite{Wei}.  

For torus knots, explicit formulas for colored \(SU(n)\) invariants
go back to Rosso--Jones \cite{RossoJones}.
Our argument uses the Hecke-algebraic formula of Lin--Zheng
\cite{LinZheng}.

The present paper gives a direct proof of the Chen--Liu--Zhu expansion
for the torus knot  \(T(2,2p+1)\) together with a
Bailey operator formula for every cyclotomic coefficient.  Set
\(Q=q^2,\)
and, for \(c\neq0\), define the lower-triangular Bailey operator
\begin{equation}\label{eq:intro-T-def}
(\mathcal T_cf)_m
=
\sum_{j=0}^{m}
\frac{c^jQ^{j^2}}{(Q;Q)_{m-j}}f_j.
\end{equation}
Its diagonal entries are \(c^mQ^{m^2}\), so \(\mathcal T_c\) is
invertible over the corresponding rational function field.  Our main
theorem is the following.

\begin{figure}[H]
    \includegraphics[width=0.5\textwidth]{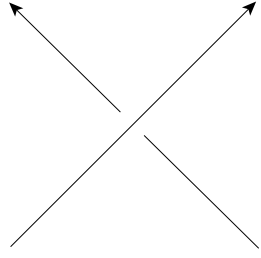}
    \caption{The standard positive generator \(\sigma_1\) of the two-strand braid group \(B_2\).}
\end{figure}

\begin{theorem}\label{thm:main}
Let \(n\geq2\), \(p\geq0\), and
\[
K_p=T(2,2p+1)=\widehat{\sigma_1^{\,2p+1}}.
\]
Then, for every \(N\geq0\),
\begin{equation}\label{eq:main-expansion}
J_N^{SU(n)}(K_p;q)
=
\sum_{m=0}^{N}
\left(
\prod_{j=0}^{m-1}\{N-j\}\{N+n+j\}
\right)
H_m^{(n,p)}(q),
\end{equation}
where the coefficients are independent of \(N\) and are given by
\begin{equation}\label{eq:main-H-operator}
H_m^{(n,p)}(q)
=
(-1)^m q^{m(m+n+1)}
\left[
\mathcal T_{Q^n}^{-(p+1)}
\mathcal T_Q^{p+1}\delta
\right]_m,
\qquad
\delta=(1,0,0,\ldots).
\end{equation}
Moreover,
\[
H_m^{(n,p)}(q)\in\mathbb Z[q^{\pm1}]
\qquad(m\geq0).
\]
For the mirror \(K_p^*\), the expansion holds with
\[
H_m^{(n)}(K_p^*;q)=H_m^{(n,p)}(q^{-1}).
\]
\end{theorem}

\begin{corollary}\label{cor:all-two-strand}
For every \(n\geq2\) and every odd integer \(k\), the colored
\(SU(n)\) invariants of the torus knot \(T(2,k)\) admit the
Chen--Liu--Zhu cyclotomic expansion.
\end{corollary}

The cyclotomic expansion also has the following immediate
consequences for congruence relations and $SU(n)$ volume conjecture; see \cite{CLZ} and Section \ref{sec:related}.

\begin{corollary}\label{cor:intro-congruence}
Let \(n\geq2\), let \(r\) be an odd integer, and let \(N\geq k\geq0\).
Then
\[
J_N^{SU(n)}(T(2,r);q)
\equiv
J_k^{SU(n)}(T(2,r);q)
\pmod{\{N-k\}\{N+k+n\}}.
\]
\end{corollary}

\begin{corollary}\label{cor:intro-volume}
Let \(n\geq2\), let \(r\) be an odd integer, and set
\[
\xi_{N,a}(s)
=
\exp\left(
\frac{s\pi\sqrt{-1}}{N+a}
\right),
\qquad a,s\in\mathbb Z.
\]
If
\[
a\notin\{1,\ldots,n-1\},
\]
then
\[
2\pi s
\lim_{N\to\infty}
\frac{
\log
J_N^{SU(n)}
\bigl(T(2,r);\xi_{N,a}(s)\bigr)
}{
N+1
}
=
0.
\]
Hence the \(a\notin\{1,\ldots,n-1\}\) part of the \(SU(n)\) volume conjecture holds for $T(2,r)$.
\end{corollary}

\begin{remark}
    Although the main cyclotomic expansion theorem is stated for
torus knots \(T(2,2p+1)\), some of the algebraic results developed in the
proof are more general.  In particular, the Newton expansion and the
uniform integral transition theorem for the Bailey operators are
independent of the specialization.
\end{remark}
\subsection{Outline}

We first derive a Newton expansion valid for arbitrary knots. Set
\[
X_N=q^{2N+n}+q^{-2N-n}.
\]
Then
\[
X_N-X_j=\{N-j\}\{N+n+j\}.
\]
Hence the cyclotomic factors in \eqref{eq:intro-C} are precisely the
Newton basis polynomials
\[
\prod_{j=0}^{m-1}(X-X_j)
\]
evaluated at \(X=X_N\).  It follows that every sequence admits a unique
Newton expansion over \(\mathbb Q(q)\).  In particular, for every knot
\(K\), the corresponding Newton coefficients are
\begin{equation}\label{eq:intro-universal-H}
H_m^{(n)}(K;q)
=
\sum_{i=0}^{m}
(-1)^{m-i}
\frac{
\{n+i-1\}!\{n+2i\}
}{
\{i\}!\{m-i\}!\{n+i+m\}!
}
J_i^{SU(n)}(K;q).
\end{equation}
Thus the independence of the color \(N\) follows from the Newton
expansion, and the nontrivial content in the cyclotomic conjecture is
the Laurent integrality of these coefficients.

The same transformation also admits a Bailey theoretic interpretation.
With \(a=Q^n\), let
\[
(M_a)_{m,i}
=
\frac{1}{(Q;Q)_{m-i}(aQ;Q)_{m+i}}.
\]
After normalization, \eqref{eq:intro-universal-H} becomes an ordinary
Bailey transform relative to \(a\).  For \(K_p=T(2,2p+1)\), the
Lin--Zheng formula leads to a well-poised Bailey transform with kernel
\[
(W_{Q,a})_{i,s}
=
\frac{(a/Q;Q)_{i-s}(a;Q)_{i+s}}
{(Q;Q)_{i-s}(Q^2;Q)_{i+s}}.
\]
The relation between the ordinary and well-poised Bailey transforms is
given by the matrix identity
\begin{equation}\label{eq:intro-WP-diagonalization}
M_aDW_{Q,a}=SM_Q,
\end{equation}
where
\[
D_i=a^{-i}Q^i\frac{1-aQ^{2i}}{1-a},
\qquad
S=\operatorname{diag}
\left(\left(\frac Qa\right)^m\right)_{m\geq0}.
\]
This identity follows from a terminating very-well-poised
\({}_6\phi_5\) summation and yields
\begin{equation}\label{eq:intro-beta-transition-new}
\beta^{(n,p)}
=
\mathcal T_{Q^n}^{-(p+1)}\mathcal T_Q^{p+1}\delta,
\qquad
H_m^{(n,p)}
=
(-1)^m q^{m(m+n+1)}\beta_m^{(n,p)}.
\end{equation}

It remains to prove the integrality of the Bailey transition in
\eqref{eq:intro-beta-transition-new}.  More generally, if \(c\) is an
indeterminate, \(\ell\in\mathbb Z\), and \(r\geq0\), then
\begin{equation}\label{eq:intro-general-transition}
\mathcal T_c^{-r}\mathcal T_{cQ^\ell}^{r}
\end{equation}
has lower-triangular matrix coefficients in
\[
\mathbb Z[Q^{\pm1},c^{\pm1}].
\]
In fact, it defines a continuous
\(\mathbb Z[Q^{\pm1},c^{\pm1}]\)-linear automorphism of
\[
\mathbb Z[Q^{\pm1},c^{\pm1}][[z]].
\]
The specialization
\[
c=Q^n,
\qquad
\ell=1-n
\]
gives the transition
\(\mathcal T_{Q^n}^{-r}\mathcal T_Q^r\)
required in \eqref{eq:intro-beta-transition-new}.  The general result is
proved in Section~\ref{sec:integrality} by factoring the Bailey operator
into a convolution operator and a diagonal operator and studying their
conjugation action on an integral algebra of formal \(q\)-difference
operators.

{\itshape
The proofs of all the results presented in this paper had already been
completed, and some of the results had been presented in academic talks,
when the author became aware of the preprint of Fang--Zhou \cite{FangZhou}, first posted on 30 August 2026.  
The manuscript was still being prepared, and the author was at the same time exploring possible extensions of the method to all torus knots and, more generally,
to arbitrary knots.  Fang--Zhou prove the Chen--Liu--Zhu conjecture for
arbitrary knots by a different approach.  The present results were
obtained independently and provide a \(q\)-hypergeometric and Bailey theoretic description of the cyclotomic coefficients for torus knots.
}

\subsection{Organization}
Section~\ref{sec:preliminaries} introduces the notation and establishes the basic \(q\)-series identities.
Section~\ref{sec:QDD} identifies the cyclotomic factors with a Newton basis.
Section~\ref{sec:torus} treats the case of two-strand torus knots using the Lin--Zheng formula.
Section~\ref{sec:bailey} rewrites the Newton transform in terms of ordinary and well-poised Bailey transforms.
Section~\ref{sec:diagonalization} proves the WP-Bailey diagonalization and reduces the problem to an ordinary Bailey transition.
Section~\ref{sec:integrality} proves the uniform integral transition theorem and completes the proof of Theorem~\ref{thm:main}.
Section~\ref{sec:related} discusses higher-strand torus knots, Bailey structures for arbitrary knots, congruence relations, and the \(SU(n)\) volume conjecture.

\section{Notation and \(q\)-series identities}
\label{sec:preliminaries}

Throughout the paper, \(q\) is an indeterminate and \(Q=q^2\).
Unless another coefficient ring is specified, identities involving
\(q\) are understood in \(\Q(q)\).  All power series identities are understood formally, and no analytic
convergence is involved.

For \(r\in\mathbb Z\), we write
\[
\{r\}=q^r-q^{-r}.
\]

For \(r\geq0\), set
\[
\curly{r}!
=
\prod_{j=1}^{r}\curly{j},
\qquad
\curly{0}!=1.
\]
And the \(Q\)-Pochhammer symbol is
\[
(x;Q)_r
=
\prod_{j=0}^{r-1}(1-xQ^j),
\qquad
(x;Q)_0=1.
\]
We also use
\[
(x_1,\ldots,x_s;Q)_r
=
\prod_{\nu=1}^{s}(x_\nu;Q)_r.
\]

The Gaussian binomial coefficient is
\[
\Gauss{M}{r}
=
\frac{(Q;Q)_M}
{(Q;Q)_r(Q;Q)_{M-r}}
\qquad
(0\leq r\leq M).
\]

\begin{lemma}\label{lem:q-factorial-conversion}
For every \(r\geq0\),
\begin{equation}\label{eq:q-factorial-conversion}
\curly{r}!
=
(-1)^r
q^{-r(r+1)/2}
(Q;Q)_r.
\end{equation}
More generally, if \(u,k\geq0\), then
\begin{equation}\label{eq:factorial-ratio}
\frac{\curly{u+k}!}{\curly{u}!}
=
(-1)^k
q^{-ku-k(k+1)/2}
(Q^{u+1};Q)_k.
\end{equation}
\end{lemma}

\begin{proof}
Since
\[
\{j\}=-q^{-j}(1-Q^j),
\]
we have
\[
\{r\}!
=
(-1)^r q^{-\sum_{j=1}^r j}
\prod_{j=1}^r(1-Q^j)
=
(-1)^r q^{-r(r+1)/2}(Q;Q)_r.
\]
Similarly,
\[
\frac{\{u+k\}!}{\{u\}!}
=
\prod_{j=1}^k\{u+j\}
=
(-1)^k
q^{-ku-k(k+1)/2}
(Q^{u+1};Q)_k.
\]
\end{proof}

We use the \(q\)-binomial identity \cite{GasperRahman}

\begin{equation}\label{eq:negative-q-binomial}
\frac{1}{(x;Q)_d}
=
\sum_{r=0}^{\infty}
\Gauss{d+r-1}{r}
x^r,
\qquad d\geq1.
\end{equation}

We also use the terminating very-well-poised \({}_6\phi_5\) summation \cite{GasperRahman}:
\begin{align}
&\sum_{r=0}^{N}
\frac{1-AQ^{2r}}{1-A}
\frac{
(A,b,c,Q^{-N};Q)_r
}{
(Q,AQ/b,AQ/c,AQ^{N+1};Q)_r
}
\left(
\frac{AQ^{N+1}}{bc}
\right)^r
\nonumber\\
&\hspace{35mm}
=
\frac{
(AQ,AQ/(bc);Q)_N
}{
(AQ/b,AQ/c;Q)_N
}.
\label{eq:6phi5}
\end{align}

\begin{lemma}
\label{lem:degenerate-6phi5}
For \(N\geq0\), the following identity holds in
\(\Q(A,b,Q)\):
\begin{align}
&\sum_{r=0}^{N}
\frac{1-AQ^{2r}}{1-A}
\frac{(A,b,Q^{-N};Q)_r}
{(Q,AQ/b,AQ^{N+1};Q)_r}
(-1)^r b^{-r}Q^{Nr-\binom r2}
\nonumber\\
&\hspace{35mm}
=
b^{-N}
\frac{(AQ;Q)_N}{(AQ/b;Q)_N}.
\label{eq:degenerate-6phi5}
\end{align}
\end{lemma}

\begin{proof}
Since \eqref{eq:6phi5} is finite, we may evaluate it at \(c=0\) after cancellation.
\[
\left.
\frac{(c;Q)_r}{(AQ/c;Q)_r}
\left(\frac{AQ^{N+1}}{bc}\right)^r
\right|_{c=0}
=
(-1)^r b^{-r}Q^{Nr-\binom r2}.
\]
Similarly,
\[
\left.
\frac{(AQ/(bc);Q)_N}{(AQ/c;Q)_N}
\right|_{c=0}
=
b^{-N}.
\]
Substitution into \eqref{eq:6phi5} gives
\eqref{eq:degenerate-6phi5}.
\end{proof}

\begin{remark}
Equation~\eqref{eq:negative-q-binomial} is an identity in the formal
power series ring \(\Q(Q)[[x]]\), while \eqref{eq:6phi5} and
\eqref{eq:degenerate-6phi5} are finite rational function identities.
\end{remark}

\section{Cyclotomic factors as a Newton basis}
\label{sec:QDD}

\subsection{Newton expansions}

Fix \(n\geq2\).  Define
\begin{equation}\label{eq:XN}
X_N
=
q^{2N+n}+q^{-2N-n},
\qquad N\geq0.
\end{equation}

The basic observation is the following.

\begin{lemma}\label{lem:newton-difference}
For all \(N,j\geq0\),
\begin{equation}\label{eq:X-difference}
X_N-X_j
=
\curly{N-j}\curly{N+n+j}.
\end{equation}
\end{lemma}

\begin{proof}
A direct calculation gives
\begin{align*}
\{N-j\}\{N+n+j\}
&=(q^{N-j}-q^{-N+j})(q^{N+n+j}-q^{-N-n-j})\\
&=q^{2N+n}+q^{-2N-n}
  -q^{2j+n}-q^{-2j-n}\\
&=X_N-X_j.
\end{align*}
\end{proof}

It follows immediately that
\begin{equation}\label{eq:C-newton}
C_{N+1,m}^{(n)}
=
\prod_{j=0}^{m-1}(X_N-X_j).
\end{equation}

For \(m\geq0\), set
\[
P_m(X)=\prod_{j=0}^{m-1}(X-X_j),
\]
with \(P_0(X)=1\).  Then \eqref{eq:C-newton} reads
\[
C_{N+1,m}^{(n)}=P_m(X_N).
\]
Thus \(P_0,P_1,\ldots\) form the Newton basis determined by the
sequence \(X_0,X_1,\ldots\), and the Chen--Liu--Zhu factors are its
values at \(X_N\).

\begin{proposition}\label{prop:newton-general}
Let \(F_0,F_1,\ldots\) be any sequence in \(\Q(q)\).
There exists a unique sequence
\[
A_0,A_1,\ldots\in\Q(q)
\]
such that
\begin{equation}\label{eq:newton-expansion-general}
F_N
=
\sum_{m=0}^{N}
A_m
\prod_{j=0}^{m-1}(X_N-X_j)
\qquad
(N\geq0).
\end{equation}
Moreover,
\begin{equation}\label{eq:newton-divided-difference}
A_m
=
\sum_{i=0}^{m}
\frac{F_i}
{\displaystyle
\prod_{\substack{0\leq j\leq m\\j\neq i}}
(X_i-X_j)}.
\end{equation}
\end{proposition}

\begin{proof}
The \(X_N\) are pairwise distinct by
Lemma~\ref{lem:newton-difference}.  The expansion
\eqref{eq:newton-expansion-general} is therefore determined recursively:
evaluation at \(N=0,1,\ldots\) determines \(A_0,A_1,\ldots\)
successively, since
\[
\prod_{j=0}^{m-1}(X_m-X_j)\neq0.
\]
This proves existence and uniqueness.
For fixed \(m\), set
\[
I_m(X)
=
\sum_{r=0}^{m}
A_r\prod_{j=0}^{r-1}(X-X_j).
\]
By the recursive construction,
\[
I_m(X_i)=F_i
\qquad (0\leq i\leq m).
\]
Hence \(I_m\) is the unique interpolation polynomial of degree at
most \(m\) through the points \((X_i,F_i)\). Its Lagrange formula is
\[
I_m(X)
=
\sum_{i=0}^{m}
F_i
\prod_{\substack{0\leq j\leq m\\ j\neq i}}
\frac{X-X_j}{X_i-X_j}.
\]
Since
\(
\prod_{j=0}^{m-1}(X-X_j)
\)
is monic of degree \(m\), \(A_m\) is the leading coefficient of
\(I_m\). Comparing leading coefficients gives
\[
A_m
=
\sum_{i=0}^{m}
\frac{F_i}
{\displaystyle
\prod_{\substack{0\leq j\leq m\\ j\neq i}}
(X_i-X_j)}.
\]
\end{proof}

We next calculate the divided-difference denominator.

\begin{lemma}\label{lem:newton-denominator}
For \(0\leq i\leq m\),
\begin{equation}\label{eq:newton-denominator}
\prod_{\substack{0\leq j\leq m\\j\neq i}}
(X_i-X_j)
=
(-1)^{m-i}
\frac{
\curly{i}!\curly{m-i}!\curly{n+i+m}!
}{
\curly{n+i-1}!\curly{n+2i}
}.
\end{equation}
\end{lemma}
\begin{proof}
By Lemma~\ref{lem:newton-difference},
\[
\prod_{\substack{0\leq j\leq m\\j\neq i}}
(X_i-X_j)
=
\left(\prod_{j\neq i}\{i-j\}\right)
\left(\prod_{j\neq i}\{n+i+j\}\right).
\]
The two factors are
\[
\prod_{j\neq i}\{i-j\}
=
(-1)^{m-i}\{i\}!\{m-i\}!,
\]
and
\[
\prod_{j\neq i}\{n+i+j\}
=
\frac{\{n+i+m\}!}
{\{n+i-1\}!\{n+2i\}}.
\]
Multiplying them gives \eqref{eq:newton-denominator}.
\end{proof}
Applying Proposition~\ref{prop:newton-general} to the colored
\(SU(n)\) invariant gives the following universal formula.

\begin{proposition}\label{prop:H-universal}
For every knot \(K\), there exists a unique sequence
\[
H_m^{(n)}(K;q)\in\Q(q),
\qquad m\geq0,
\]
such that
\[
J_N^{SU(n)}(K;q)
=
\sum_{m=0}^{N}
C_{N+1,m}^{(n)}
H_m^{(n)}(K;q).
\]
It is given explicitly by
\begin{equation}\label{eq:H-universal}
H_m^{(n)}(K;q)
=
\sum_{i=0}^{m}
(-1)^{m-i}
\frac{
\curly{n+i-1}!\curly{n+2i}
}{
\curly{i}!\curly{m-i}!\curly{n+i+m}!
}
J_i^{SU(n)}(K;q).
\end{equation}
\end{proposition}

\begin{proof}
Take
\[
F_N=J_N^{SU(n)}(K;q)
\]
in Proposition~\ref{prop:newton-general}. By \eqref{eq:C-newton}, the Newton factors evaluated at \(X_N\)
are precisely \(C_{N+1,m}^{(n)}\). 
Substituting \eqref{eq:newton-denominator} into \eqref{eq:newton-divided-difference} gives \eqref{eq:H-universal}.
\end{proof}

\begin{remark}
Over \(\Q(q)\), existence and uniqueness are therefore apparent.
The nontrivial content of the cyclotomic expansion conjecture is the
stronger assertion
\[
H_m^{(n)}(K;q)\in\Z[q^{\pm1}].
\]
\end{remark}

\subsection{Mirror symmetry}
\label{sec:mirror}

Under our normalization, the mirror knot \(K^*\) satisfies
\[
J_N^{SU(n)}(K^*;q)
=
J_N^{SU(n)}(K;q^{-1}).
\]

\begin{lemma}\label{lem:mirror-basis}
The cyclotomic basis is invariant under \(q\mapsto q^{-1}\):
\[
C_{N+1,m}^{(n)}(q^{-1})
=
C_{N+1,m}^{(n)}(q).
\]
\end{lemma}

\begin{proof}
Since
\[
\{r\}_{q^{-1}}=-\{r\}_q,
\]
each product
\[
\{N-j\}\{N+n+j\}
\]
is invariant under \(q\mapsto q^{-1}\).  Hence so is
\(C_{N+1,m}^{(n)}\).
\end{proof}

\begin{corollary}\label{cor:mirror}
If
\[
J_N^{SU(n)}(K;q)
=
\sum_{m=0}^{N}
C_{N+1,m}^{(n)}(q)
H_m^{(n)}(K;q),
\]
then
\[
H_m^{(n)}(K^*;q)
=
H_m^{(n)}(K;q^{-1}).
\]
In particular, Laurent integrality of the coefficients for \(K\)
is equivalent to Laurent integrality for \(K^*\).
\end{corollary}

\begin{proof}
Replacing \(q\) by \(q^{-1}\) in the expansion for \(K\) gives
\[
J_N^{SU(n)}(K^*;q)
=
\sum_{m=0}^{N}
C_{N+1,m}^{(n)}(q^{-1})
H_m^{(n)}(K;q^{-1}).
\]
By Lemma~\ref{lem:mirror-basis},
\[
J_N^{SU(n)}(K^*;q)
=
\sum_{m=0}^{N}
C_{N+1,m}^{(n)}(q)
H_m^{(n)}(K;q^{-1}).
\]
The uniqueness in Proposition~\ref{prop:newton-general} now gives
\[
H_m^{(n)}(K^*;q)=H_m^{(n)}(K;q^{-1}).
\]
\end{proof}

\section{Colored \(SU(n)\) invariants of \(T(2,2p+1)\)}
\label{sec:torus}

Throughout this section,
\[
K_p=T(2,2p+1)=\widehat{\sigma_1^{\,2p+1}}
\]
is equipped with zero framing, and the reduced invariant is normalized
as in the Introduction.  Write
\[
k=2p+1.
\]

\subsection{A plethysm identity}

Let \(s_\lambda\) denote the Schur function indexed by a partition \(\lambda\), and write \(h_r=s_{(r)}\).
We use the standard conventions of \cite{Macdonald}.

\begin{lemma}\label{lem:plethysm}
For every \(N\geq0\),
\begin{equation}\label{eq:plethysm}
s_{(N)}(x_1^2,x_2^2,\ldots)
=
\sum_{j=0}^{N}
(-1)^j
s_{(2N-j,j)}(x_1,x_2,\ldots).
\end{equation}
\end{lemma}

\begin{proof}
The Jacobi--Trudi identity gives
\[
s_{(2N-j,j)}
=
h_{2N-j}h_j-h_{2N-j+1}h_{j-1},
\]
where \(h_{-1}=0\).
Therefore
\begin{align*}
\sum_{j=0}^{N}(-1)^j s_{(2N-j,j)}
&=
\sum_{j=0}^{N}
(-1)^j h_{2N-j}h_j
-
\sum_{j=1}^{N}
(-1)^j h_{2N-j+1}h_{j-1}
\\
&=
\sum_{j=0}^{N}
(-1)^j h_{2N-j}h_j
+
\sum_{j=0}^{N-1}
(-1)^j h_{2N-j}h_j.
\end{align*}
Pairing the \(j\)-th and
\((2N-j)\)-th terms shows that this is exactly
\[
\sum_{j=0}^{2N}
(-1)^j h_{2N-j}h_j.
\]

Let
\[
H(u)=\sum_{r\geq0}h_r u^r
=
\prod_{\ell\geq1}(1-x_\ell u)^{-1}.
\]
Then
\[
H(u)H(-u)
=
\prod_{\ell\geq1}
\frac{1}{(1-x_\ell u)(1+x_\ell u)}
=
\prod_{\ell\geq1}
\frac{1}{1-x_\ell^2u^2}.
\]
Consequently,
\[
H(u)H(-u)
=
\sum_{N\geq0}
h_N(x_1^2,x_2^2,\ldots)u^{2N}.
\]
The coefficient of \(u^{2N}\) in the left-hand side is
\[
\sum_{j=0}^{2N}
(-1)^j h_{2N-j}h_j.
\]
Thus
\[
\sum_{j=0}^{N}
(-1)^j s_{(2N-j,j)}
=
h_N(x_1^2,x_2^2,\ldots)
=
s_{(N)}(x_1^2,x_2^2,\ldots).
\]
\end{proof}

\subsection{Quantum dimensions for the two-row partitions}

For a partition \(\lambda\) with at most \(n\) parts, the corresponding
quantum dimension is
\[
\dim_q V_\lambda
=
s_\lambda(q^{n-1},q^{n-3},\ldots,q^{-(n-1)}).
\]

The hook-content formula gives (see \cite{Macdonald}),
\begin{equation}\label{eq:quantum-hook}
\dim_qV_\lambda
=
\prod_{(r,c)\in\lambda}
\frac{\curly{n+c-r}}{\curly{h_{r,c}}},
\end{equation}
where \(h_{r,c}\) is the hook length of the box \((r,c)\).

\begin{lemma}\label{lem:two-row-dimension}
Let
\[
\lambda=(2N-j,j),
\qquad
0\leq j\leq N.
\]
Then
\begin{equation}\label{eq:two-row-dim}
\dim_qV_{(2N-j,j)}
=
\frac{
\curly{2N-2j+1}
\curly{n+2N-j-1}!
\curly{n+j-2}!
}{
\curly{2N-j+1}!
\curly{j}!
\curly{n-1}!
\curly{n-2}!
}.
\end{equation}
Moreover,
\begin{equation}\label{eq:symmetric-dim}
\dim_qV_{(N)}
=
\frac{
\curly{n+N-1}!
}{
\curly{N}!\curly{n-1}!
}.
\end{equation}
\end{lemma}

\begin{proof}
Put
\[
A=2N-j,\qquad B=j.
\]
The numerator in the hook-content formula is
\begin{align*}
\prod_{c=1}^{A}\curly{n+c-1}
\prod_{c=1}^{B}\curly{n+c-2}
&=
\frac{\curly{n+A-1}!}{\curly{n-1}!}
\frac{\curly{n+B-2}!}{\curly{n-2}!}.
\end{align*}

We next calculate the hook-length product.
For boxes in the first row with \(1\leq c\leq B\),
\[
h_{1,c}=A-c+2.
\]
For \(B<c\leq A\),
\[
h_{1,c}=A-c+1.
\]
For the second row,
\[
h_{2,c}=B-c+1.
\]
Hence
\begin{align*}
\prod_{(r,c)\in(A,B)}\curly{h_{r,c}}
&=
\left(
\prod_{c=1}^{B}\curly{A-c+2}
\right)
\left(
\prod_{c=B+1}^{A}\curly{A-c+1}
\right)
\curly{B}!
\\
&=
\frac{\curly{A+1}!}{\curly{A-B+1}!}
\curly{A-B}!
\curly{B}!
\\
&=
\frac{
\curly{A+1}!\curly{B}!
}{
\curly{A-B+1}
}.
\end{align*}
Substituting
\[
A=2N-j,\qquad B=j
\]
into \eqref{eq:quantum-hook} gives
\eqref{eq:two-row-dim}.

For \(\lambda=(N)\), the numerator is
\[
\frac{\curly{n+N-1}!}{\curly{n-1}!},
\]
while the hook-length product is \(\curly{N}!\).
This proves \eqref{eq:symmetric-dim}.
\end{proof}

Define
\begin{equation}\label{eq:R-def}
R_{N,j}^{(n)}
=
\frac{
\dim_qV_{(2N-j,j)}
}{
\dim_qV_{(N)}
}.
\end{equation}

Combining the two formulas in
Lemma~\ref{lem:two-row-dimension} gives
\begin{equation}\label{eq:R-curly}
R_{N,j}^{(n)}
=
\frac{
\curly{2N-2j+1}
\curly{n+2N-j-1}!
\curly{n+j-2}!
\curly{N}!
}{
\curly{2N-j+1}!
\curly{j}!
\curly{n-2}!
\curly{n+N-1}!
}.
\end{equation}

\subsection{The Lin--Zheng formula}

We now derive the exact formula for the colored \(SU(n)\) invariant.

The variables \(t_{\rm LZ},\nu_{\rm LZ}\) used by Lin--Zheng are
related to the variables used here by
\[
t_{\rm LZ}^{1/2}=q^{-1},
\qquad
\nu_{\rm LZ}^{1/2}=q^{-n}.
\]
Thus
\[
t_{\rm LZ}=q^{-2},
\qquad
\nu_{\rm LZ}=q^{-2n}.
\]
In this specialization, the reduced invariant used in the present
paper is
\begin{equation}\label{eq:reduced-normalization}
J_N^{SU(n)}(K_p;q)
=
\frac{W_{K_p;(N)}(q^{-2},q^{-2n})}
{s_{(N)}^*(q^{-2},q^{-2n})}
=
\frac{W_{K_p;(N)}(q^{-2},q^{-2n})}{\dim_qV_{(N)}}.
\end{equation}

For a partition \(\lambda\), write
\[
\kappa_\lambda
=
\sum_{i\geq1}\lambda_i(\lambda_i-2i+1).
\]
Since \(k=2p+1\) is odd, \(\gcd(2,k)=1\).  Thus
specialized to \(r=2\), and color \((N)\), \cite{LinZheng} gives
\begin{align}
W_{K_p;(N)}
&=
t_{\rm LZ}^{\,k\kappa_{(N)}}
\nu_{\rm LZ}^{\,kN/2}
\sum_{\lambda\vdash2N}
c_{(N)}^\lambda
t_{\rm LZ}^{-k\kappa_\lambda/4}
s_\lambda^*(t_{\rm LZ},\nu_{\rm LZ}).
\label{eq:LZ-special}
\end{align}
where the coefficients \(c_{(N)}^\lambda\) are defined by
\[
s_{(N)}(x_1^2,x_2^2,\ldots)
=
\sum_{\lambda\vdash 2N}
c_{(N)}^\lambda s_\lambda(x_1,x_2,\ldots).
\]
By Lemma~\ref{lem:plethysm} and the linear independence of the Schur
functions,
\[
c_{(N)}^{(2N-j,j)}=(-1)^j,
\qquad 0\leq j\leq N,
\]
and all other coefficients vanish.

\begin{proposition}\label{prop:torus-J}
For \(K_p=T(2,2p+1)\),
\begin{equation}\label{eq:torus-J}
J_N^{SU(n)}(K_p;q)
=
\sum_{j=0}^{N}
(-1)^j
q^{(2p+1)
\left(
j^2-(2N+1)j-(n-1)N
\right)}
R_{N,j}^{(n)},
\end{equation}
where \(R_{N,j}^{(n)}\) is given by
\eqref{eq:R-curly}.
\end{proposition}

\begin{proof}
Let \(k=2p+1\).
For the partition \((N)\),
\[
\kappa_{(N)}
=
N(N-1).
\]
After the specialization
\[
t_{\rm LZ}=q^{-2},
\qquad
\nu_{\rm LZ}=q^{-2n},
\]
the prefactor in \eqref{eq:LZ-special} becomes
\[
t_{\rm LZ}^{k\kappa_{(N)}}
\nu_{\rm LZ}^{kN/2}
=
q^{-2kN(N-1)}
q^{-nkN}.
\]
Furthermore,
\[
t_{\rm LZ}^{-k\kappa_\lambda/4}
=
q^{k\kappa_\lambda/2}.
\]
The specialized Schur function is the quantum dimension
\[
s_\lambda^*
=
s_\lambda(q^{n-1},q^{n-3},\ldots,q^{-(n-1)})
=
\dim_qV_\lambda.
\]

For
\[
\lambda=(2N-j,j),
\]
we have
\begin{align*}
\kappa_\lambda
&=
(2N-j)(2N-j-1)+j(j-3)
\\
&=
4N^2-4Nj+2j^2-2N-2j.
\end{align*}
Therefore
\begin{align*}
\frac{\kappa_\lambda}{2}
-2N(N-1)-nN
&=
\left(
2N^2-2Nj+j^2-N-j
\right)
\\
&\qquad
-2N^2+2N-nN
\\
&=
j^2-(2N+1)j-(n-1)N.
\end{align*}

Consequently,
\[
W_{K_p;(N)}
=
\sum_{j=0}^{N}
(-1)^j
q^{k(j^2-(2N+1)j-(n-1)N)}
\dim_qV_{(2N-j,j)}.
\]
By the normalization \eqref{eq:reduced-normalization},
\[
J_N^{SU(n)}(K_p;q)
=
\sum_{j=0}^{N}
(-1)^j
q^{k(j^2-(2N+1)j-(n-1)N)}
\frac{\dim_qV_{(2N-j,j)}}{\dim_qV_{(N)}}.
\]
Using \eqref{eq:R-def} proves the proposition.
\end{proof}

\subsection{Newton coefficients for the torus knots}
\label{sec:explicit-H}

Substituting Proposition~\ref{prop:torus-J} into the universal
Newton formula gives an expression for the desired coefficient.

\begin{proposition}\label{prop:H-double-sum}
For \(m\geq0\),
\begin{align}
H_m^{(n,p)}
={}&
\sum_{i=0}^{m}\sum_{j=0}^{i}
(-1)^{m-i+j}
q^{(2p+1)
\left(
j^2-(2i+1)j-(n-1)i
\right)}
\nonumber\\
&\quad\times
\frac{
\curly{n+2i}
\curly{2i-2j+1}
\curly{n+2i-j-1}!
\curly{n+j-2}!
}{
\curly{m-i}!
\curly{n+i+m}!
\curly{2i-j+1}!
\curly{j}!
\curly{n-2}!
}.
\label{eq:H-double-sum}
\end{align}
\end{proposition}

\begin{proof}
Substituting \eqref{eq:torus-J} with \(N=i\) into
\eqref{eq:H-universal} gives
\[
H_m^{(n,p)}
=
\sum_{i=0}^{m}\sum_{j=0}^{i}
(-1)^{m-i+j}
q^{(2p+1)(j^2-(2i+1)j-(n-1)i)}
\frac{\{n+i-1\}!\{n+2i\}}
{\{i\}!\{m-i\}!\{n+i+m\}!}
R_{i,j}^{(n)}.
\]
Substituting \eqref{eq:R-curly} and cancelling
\(\{i\}!\) and \(\{n+i-1\}!\) gives
\eqref{eq:H-double-sum}.
\end{proof}

The Laurent integrality of these coefficients is not apparent from
\eqref{eq:H-double-sum}.  We next reinterpret this expression in terms
of Bailey transforms.

\section{Bailey transforms for the torus knot coefficients}
\label{sec:bailey}

\subsection{The Newton transform as an ordinary Bailey transform}

Recall that \(Q=q^2\), and set
\[
a=Q^n.
\]

All matrices in this section are indexed by \(\mathbb Z_{\geq0}\) and
are lower triangular, so all matrix products below are well defined
entrywise by finite sums.

For \(c\) an indeterminate, define the ordinary Bailey matrix
relative to \(c\) by
\begin{equation}\label{eq:Bailey-matrix}
(M_c)_{m,i}
=
\frac{1}{
(Q;Q)_{m-i}(cQ;Q)_{m+i}},
\qquad
0\leq i\leq m.
\end{equation}

Thus a pair of sequences \((\alpha_i,\beta_i)\) is a Bailey pair
relative to \(c\), with base \(Q\), if
\[
\beta=M_c\alpha.
\]
This is the standard Bailey-pair normalization originating in
Bailey's work \cite{Bailey}; see also \cite{AndrewsQSeries}.  Equivalently,
\begin{equation}\label{eq:Bailey-pair}
\beta_m
=
\sum_{i=0}^{m}
\frac{\alpha_i}{
(Q;Q)_{m-i}(cQ;Q)_{m+i}}.
\end{equation}

We first rewrite the universal Newton kernel.

\begin{lemma}\label{lem:newton-Q-kernel}
For \(0\leq i\leq m\),
\begin{align}
&(-1)^{m-i}
\frac{
\curly{n+i-1}!\curly{n+2i}
}{
\curly{i}!\curly{m-i}!\curly{n+i+m}!
}
\nonumber\\
&\qquad=
(-1)^{m-i}
q^{m(m+n+1)+i(i-1)}
\frac{
(a;Q)_i(1-aQ^{2i})
}{
(Q;Q)_i(Q;Q)_{m-i}(a;Q)_{m+i+1}
}.
\label{eq:newton-Q-kernel}
\end{align}
\end{lemma}

\begin{proof}
Apply Lemma~\ref{lem:q-factorial-conversion} separately to all
quantum factorials.

First,
\[
\frac{\curly{n+i-1}!}{\curly{i}!}
=
(-1)^{n-1}
q^{-\frac{(n+i-1)(n+i)}2+\frac{i(i+1)}2}
\frac{(Q;Q)_{n+i-1}}{(Q;Q)_i}.
\]
Also,
\[
\frac{1}{\curly{m-i}!}
=
(-1)^{m-i}
q^{(m-i)(m-i+1)/2}
\frac{1}{(Q;Q)_{m-i}},
\]
and
\[
\frac{1}{\curly{n+i+m}!}
=
(-1)^{n+i+m}
q^{(n+i+m)(n+i+m+1)/2}
\frac{1}{(Q;Q)_{n+i+m}}.
\]
Finally,
\[
\curly{n+2i}
=
-q^{-n-2i}(1-Q^{n+2i})
=
-q^{-n-2i}(1-aQ^{2i}).
\]

Collecting all signs gives the sign displayed on the right-hand side
of \eqref{eq:newton-Q-kernel}.
For the \(q\)-power, direct simplification yields
\[
m(m+n+1)+i(i-1).
\]

For the \(Q\)-products, note that
\[
\frac{(Q;Q)_{n+i-1}}{(Q;Q)_i}
\cdot
\frac{1}{(Q;Q)_{n+i+m}}
=
\frac{(a;Q)_i}{
(Q;Q)_i(a;Q)_{m+i+1}
}.
\]
Indeed,
\[
(a;Q)_i
=
(Q^n;Q)_i
=
\frac{(Q;Q)_{n+i-1}}{(Q;Q)_{n-1}},
\]
whereas
\[
(a;Q)_{m+i+1}
=
\frac{(Q;Q)_{n+m+i}}{(Q;Q)_{n-1}}.
\]
Substitution gives \eqref{eq:newton-Q-kernel}.
\end{proof}

\begin{theorem}\label{thm:universal-bailey}
For an arbitrary knot \(K\), define
\begin{equation}\label{eq:alpha-universal}
\alpha_i^{(n)}(K)
=
(-1)^i
Q^{\binom{i}{2}}
\frac{1-aQ^{2i}}{1-a}
\frac{(a;Q)_i}{(Q;Q)_i}
J_i^{SU(n)}(K;q),
\end{equation}
and
\begin{equation}\label{eq:beta-universal}
\beta_m^{(n)}(K)
=
(-1)^m
q^{-m(m+n+1)}
H_m^{(n)}(K;q).
\end{equation}
Then
\[
\beta_m^{(n)}(K)
=
\sum_{i=0}^{m}
\frac{
\alpha_i^{(n)}(K)
}{
(Q;Q)_{m-i}(aQ;Q)_{m+i}
}.
\]
Equivalently,
\[
\beta^{(n)}(K)
=
M_a\alpha^{(n)}(K).
\]
\end{theorem}

\begin{proof}
Starting from \eqref{eq:H-universal}, insert
\eqref{eq:newton-Q-kernel}.  We obtain
\begin{align*}
H_m^{(n)}(K;q)
&=
(-1)^m
q^{m(m+n+1)}
\\
&\quad\times
\sum_{i=0}^{m}
\frac{
(-1)^iQ^{\binom{i}{2}}
(a;Q)_i(1-aQ^{2i})
J_i^{SU(n)}(K;q)
}{
(Q;Q)_i(Q;Q)_{m-i}(a;Q)_{m+i+1}
}.
\end{align*}
Since
\[
(a;Q)_{m+i+1}
=
(1-a)(aQ;Q)_{m+i},
\]
the summand becomes
\[
\frac{
(-1)^iQ^{\binom{i}{2}}
\frac{1-aQ^{2i}}{1-a}
\frac{(a;Q)_i}{(Q;Q)_i}
J_i^{SU(n)}(K;q)
}{
(Q;Q)_{m-i}(aQ;Q)_{m+i}
}.
\]
The numerator is precisely \(\alpha_i^{(n)}(K)\).
Multiplying by
\[
(-1)^m q^{-m(m+n+1)}
\]
and using \eqref{eq:beta-universal} gives
\[
\beta^{(n)}(K)=M_a\alpha^{(n)}(K).
\]
\end{proof}

Thus the universal Newton transform is an ordinary Bailey transform
relative to \(a=Q^n\).

\subsection{The torus-knot \(\alpha\)-sequence}
\label{sec:torus-alpha}

We next rewrite \(R_{i,j}^{(n)}\) in \(Q\)-notation.

\begin{lemma}\label{lem:R-Q}
Let \(a=Q^n\).  Then
\begin{equation}\label{eq:R-Q}
R_{i,j}^{(n)}
=
q^{-ni+i+2j}
\frac{
(1-Q^{2i-2j+1})
(aQ^i;Q)_{i-j}
(a/Q;Q)_j
(Q;Q)_i
}{
(Q;Q)_{2i-j+1}
(Q;Q)_j
}.
\end{equation}
\end{lemma}

\begin{proof}
Starting with \eqref{eq:R-curly}, write
\begin{align*}
R_{i,j}^{(n)}
&=
\curly{2i-2j+1}
\frac{\curly{n+2i-j-1}!}{\curly{n+i-1}!}
\frac{\curly{n+j-2}!}{\curly{n-2}!}
\frac{\curly{i}!}{
\curly{2i-j+1}!\curly{j}!
}.
\end{align*}

By \eqref{eq:factorial-ratio},
\begin{align*}
\frac{\curly{n+2i-j-1}!}{\curly{n+i-1}!}
&=
(-1)^{i-j}
q^{-(i-j)(n+i-1)
-(i-j)(i-j+1)/2}
(aQ^i;Q)_{i-j},
\\
\frac{\curly{n+j-2}!}{\curly{n-2}!}
&=
(-1)^j
q^{-j(n-2)-j(j+1)/2}
(a/Q;Q)_j.
\end{align*}
Furthermore,
\[
\curly{2i-2j+1}
=
-q^{-(2i-2j+1)}
(1-Q^{2i-2j+1}),
\]
and Lemma~\ref{lem:q-factorial-conversion} converts the remaining
factorials.

After cancellation of all signs, the \(Q\)-Pochhammer factors are
exactly
\[
\frac{
(1-Q^{2i-2j+1})
(aQ^i;Q)_{i-j}
(a/Q;Q)_j
(Q;Q)_i
}{
(Q;Q)_{2i-j+1}
(Q;Q)_j
}.
\]
The total \(q\)-exponent simplifies to
\[
-ni+i+2j.
\]
This proves \eqref{eq:R-Q}.
\end{proof}

We now insert \eqref{eq:R-Q} into the universal \(\alpha\)-sequence. For brevity, write
\[
\alpha_i^{(n,p)}
:=
\alpha_i^{(n)}(K_p).
\]

\begin{proposition}\label{prop:torus-alpha}
For \(K_p=T(2,2p+1)\),
\begin{align}
\alpha_i^{(n,p)}
={}&
a^{-(p+1)i}
Q^{-pi^2+i}
\frac{1-aQ^{2i}}{1-a}
(a;Q)_i
\nonumber\\
&\times
\sum_{s=0}^{i}
(-1)^s
Q^{ps(s+1)+\binom{s}{2}}
\frac{
(1-Q^{2s+1})
(aQ^i;Q)_s
(a/Q;Q)_{i-s}
}{
(Q;Q)_{i+s+1}
(Q;Q)_{i-s}
}.
\label{eq:torus-alpha}
\end{align}
\end{proposition}

\begin{proof}
By \eqref{eq:alpha-universal} and
Proposition~\ref{prop:torus-J},
\begin{align*}
\alpha_i^{(n,p)}
&=
(-1)^iQ^{\binom{i}{2}}
\frac{1-aQ^{2i}}{1-a}
\frac{(a;Q)_i}{(Q;Q)_i}
\\
&\quad\times
\sum_{j=0}^{i}
(-1)^j
q^{(2p+1)
(j^2-(2i+1)j-(n-1)i)}
R_{i,j}^{(n)}.
\end{align*}
Insert \eqref{eq:R-Q}.
The factor \((Q;Q)_i\) cancels.

Set
\[
s=i-j.
\]
Then \(j=i-s\), and
\[
(-1)^{i+j}
=
(-1)^{2i-s}
=
(-1)^s.
\]
Moreover,
\begin{align}
j^2-(2i+1)j-(n-1)i
&=
(i-s)^2-(2i+1)(i-s)-(n-1)i
\nonumber\\
&=
-i^2-ni+s^2+s.
\end{align}

The additional \(q\)-powers are
\[
Q^{\binom{i}{2}}
=
q^{i(i-1)}
\]
and, from \eqref{eq:R-Q},
\[
q^{-ni+i+2j}
=
q^{-ni+3i-2s}.
\]
Therefore the total exponent of \(q\) equals
\begin{align*}
&i(i-1)-ni+3i-2s
\\
&\qquad
+(2p+1)(-i^2-ni+s^2+s)
\\
&=
2\left[
-pi(i+n)-i(n-1)
+ps(s+1)+\binom{s}{2}
\right].
\end{align*}
Since \(Q=q^2\), this is
\[
Q^{-pi(i+n)-i(n-1)
+ps(s+1)+\binom{s}{2}}.
\]

Finally,
\[
Q^{-pi(i+n)-i(n-1)}
=
a^{-(p+1)i}Q^{-pi^2+i},
\]
because \(a=Q^n\).

The remaining Pochhammer factors become
\[
(aQ^i;Q)_s,
\qquad
(a/Q;Q)_{i-s},
\]
and
\[
(Q;Q)_{2i-j+1}
=
(Q;Q)_{i+s+1}.
\]
Substitution gives \eqref{eq:torus-alpha}.
\end{proof}

\subsection{The well-poised Bailey kernel}
\label{sec:wp}

Well-poised Bailey pairs and their transformations were developed in a
systematic form by Andrews--Berkovich and subsequent authors; see
\cite{AndrewsBerkovich,McLaughlin,Warnaar}. 
 We use the following standard WP-Bailey kernel.

Define
\begin{equation}\label{eq:A-p}
A_s^{(p)}
=
(-1)^s
Q^{\binom{s}{2}+ps(s+1)}
\frac{1-Q^{2s+1}}{1-Q}.
\end{equation}
Also define
\begin{equation}\label{eq:W-matrix}
(W_{Q,a})_{i,s}
=
\frac{
(a/Q;Q)_{i-s}(a;Q)_{i+s}
}{
(Q;Q)_{i-s}(Q^2;Q)_{i+s}
},
\qquad
0\leq s\leq i.
\end{equation}

In the standard notation for WP-Bailey pairs, this is the kernel with
base \(Q\) and parameters
\[
a_{\mathrm{WP}}=Q,
\qquad
k_{\mathrm{WP}}=a.
\]

Let
\begin{equation}\label{eq:B-def}
B^{(p)}
=
W_{Q,a}A^{(p)}.
\end{equation}
Thus
\begin{equation}\label{eq:B-expanded}
B_i^{(p)}
=
\sum_{s=0}^{i}
\frac{
(a/Q;Q)_{i-s}(a;Q)_{i+s}
}{
(Q;Q)_{i-s}(Q^2;Q)_{i+s}
}
A_s^{(p)}.
\end{equation}

\begin{lemma}\label{lem:alpha-B}
The torus-knot sequence in
Proposition~\ref{prop:torus-alpha} can be written as
\begin{equation}\label{eq:alpha-B}
\alpha_i^{(n,p)}
=
a^{-pi}Q^{-pi^2}\Gamma_i^{(p)},
\end{equation}
where
\begin{equation}\label{eq:Gamma}
\Gamma_i^{(p)}
=
a^{-i}Q^i
\frac{1-aQ^{2i}}{1-a}
B_i^{(p)}.
\end{equation}
\end{lemma}

\begin{proof}
By \eqref{eq:A-p},
\[
A_s^{(p)}
=
(-1)^s
Q^{ps(s+1)+\binom{s}{2}}
\frac{1-Q^{2s+1}}{1-Q}.
\]
Also,
\[
(a;Q)_{i+s}
=
(a;Q)_i(aQ^i;Q)_s
\]
and
\[
(Q^2;Q)_{i+s}
=
\frac{(Q;Q)_{i+s+1}}{1-Q}.
\]
Therefore
\begin{align*}
B_i^{(p)}
&=
(a;Q)_i
\sum_{s=0}^{i}
(-1)^s
Q^{ps(s+1)+\binom{s}{2}}
\\
&\quad\times
\frac{
(1-Q^{2s+1})
(aQ^i;Q)_s
(a/Q;Q)_{i-s}
}{
(Q;Q)_{i+s+1}
(Q;Q)_{i-s}
}.
\end{align*}
Comparing with \eqref{eq:torus-alpha} gives
\[
\alpha_i^{(n,p)}
=
a^{-(p+1)i}Q^{-pi^2+i}
\frac{1-aQ^{2i}}{1-a}
B_i^{(p)}.
\]
Factoring
\[
a^{-pi}Q^{-pi^2}
\]
leaves precisely \(\Gamma_i^{(p)}\) as in
\eqref{eq:Gamma}.
\end{proof}

\subsection{Bailey chains for the torus-knot sequence}
\label{sec:bailey-chain}

For \(c\neq0\), define
\begin{equation}\label{eq:T-def}
(\mathcal T_cf)_m
=
\sum_{j=0}^{m}
\frac{
c^jQ^{j^2}
}{
(Q;Q)_{m-j}
}
f_j.
\end{equation}

We use the following standard limiting Bailey lemma \cite{AndrewsQSeries}.
\begin{lemma}\label{lem:limiting-bailey}
Suppose that
\[
\beta=M_c\alpha.
\]
If
\[
\alpha'_j=c^jQ^{j^2}\alpha_j,
\]
then
\[
\mathcal T_c\beta=M_c\alpha'.
\]
\end{lemma}

The corresponding unit Bailey pair relative to \(c\) is
\begin{equation}\label{eq:unit-Bailey}
\alpha_j^{\rm unit}(c)
=
(-1)^j
Q^{\binom{j}{2}}
\frac{1-cQ^{2j}}{1-c}
\frac{(c;Q)_j}{(Q;Q)_j},
\qquad
\beta_j^{\rm unit}
=
\delta_{j0}.
\end{equation}

In particular, when \(c=Q\),
\[
\frac{(c;Q)_j}{(Q;Q)_j}=1,
\]
and hence
\[
\alpha_j^{\rm unit}(Q)
=
(-1)^j
Q^{\binom{j}{2}}
\frac{1-Q^{2j+1}}{1-Q}
=
A_j^{(0)}.
\]
Thus
\begin{equation}\label{eq:M-Q-A0}
M_QA^{(0)}
=
\delta.
\end{equation}

Each limiting Bailey iteration relative to \(Q\) multiplies the
\(\alpha\)-sequence by
\[
Q^jQ^{j^2}
=
Q^{j(j+1)}.
\]
Therefore the sequence defined in \eqref{eq:A-p} is precisely the
result of \(p\) iterations.

\begin{lemma}\label{lem:Ap-Bailey}
For every \(p\geq0\),
\begin{equation}\label{eq:MQAp}
M_QA^{(p)}
=
\mathcal T_Q^p\delta.
\end{equation}
\end{lemma}

\begin{proof}
For \(p=0\), this is \eqref{eq:M-Q-A0}.
Suppose it holds for \(p\).
By definition,
\[
A_j^{(p+1)}
=
Q^{j(j+1)}A_j^{(p)}
=
Q^jQ^{j^2}A_j^{(p)}.
\]
By Lemma~\ref{lem:limiting-bailey} with \(c=Q\),
\[
M_QA^{(p+1)}
=
\mathcal T_Q(M_QA^{(p)}).
\]
Using the induction hypothesis,
\[
M_QA^{(p+1)}
=
\mathcal T_Q^{p+1}\delta.
\]
\end{proof}

\section{WP-Bailey diagonalization}
\label{sec:diagonalization}

\subsection{The WP-Bailey diagonalization identity}

The following identity relates the ordinary Bailey matrices to the WP-Bailey kernel.

Define the diagonal matrices
\begin{equation}\label{eq:D-def}
D
=
\operatorname{diag}\left(
a^{-i}Q^i\frac{1-aQ^{2i}}{1-a}
\right)_{i\geq0}
\end{equation}

Define also
\begin{equation}\label{eq:S-def}
S
=
\operatorname{diag}\left(
\left(\frac Qa\right)^m
\right)_{m\geq0}.
\end{equation}

\begin{theorem}[WP-Bailey diagonalization]
\label{thm:WP-diagonalization}
Let \(Q\) and \(a\) be algebraically independent indeterminates.
Then, in the algebra of lower-triangular matrices over \(\Q(Q,a)\),
\begin{equation}\label{eq:WP-diagonalization}
M_aDW_{Q,a}
=
SM_Q.
\end{equation}

Equivalently, for \(0\leq s\leq m\),
\begin{align}
&\sum_{i=s}^{m}
\frac{
a^{-i}Q^i(1-aQ^{2i})
}{
1-a
}
\frac{
(a/Q;Q)_{i-s}(a;Q)_{i+s}
}{
(Q;Q)_{m-i}
(Q;Q)_{i-s}
(aQ;Q)_{m+i}
(Q^2;Q)_{i+s}
}
\nonumber\\
&\hspace{20mm}
=
\left(\frac Qa\right)^m
\frac{1}{
(Q;Q)_{m-s}
(Q^2;Q)_{m+s}
}.
\label{eq:WP-entry}
\end{align}
\end{theorem}

\begin{proof}
Fix \(0\leq s\leq m\).
Put
\[
N=m-s,
\qquad
r=i-s,
\]
so that
\[
0\leq r\leq N.
\]
Set
\[
A=aQ^{2s},
\qquad
b=\frac aQ.
\]

The left-hand side of \eqref{eq:WP-entry} is
\[
L_{m,s}
=
\sum_{r=0}^{N}\mathcal S_r,
\]
where substitution of \(i=s+r\) gives
\begin{align*}
\mathcal S_r
={}&
a^{-s-r}Q^{s+r}
\frac{1-aQ^{2s+2r}}{1-a}
\\
&\times
\frac{
(a/Q;Q)_r
(a;Q)_{2s+r}
}{
(Q;Q)_{N-r}
(Q;Q)_r
(aQ;Q)_{2s+N+r}
(Q^2;Q)_{2s+r}
}.
\end{align*}

Now use
\[
(a;Q)_{2s+r}
=
(a;Q)_{2s}(A;Q)_r,
\]
\[
(aQ;Q)_{2s+N+r}
=
(aQ;Q)_{2s+N}
(AQ^{N+1};Q)_r,
\]
and
\[
(Q^2;Q)_{2s+r}
=
(Q^2;Q)_{2s}
(Q^{2s+2};Q)_r.
\]
Hence
\begin{equation}\label{eq:L-factor}
L_{m,s}
=
C_{m,s}
\sum_{r=0}^{N}
\frac{
(Q/a)^r
(1-AQ^{2r})
(A;Q)_r(b;Q)_r
}{
(Q;Q)_{N-r}
(Q;Q)_r
(Q^{2s+2};Q)_r
(AQ^{N+1};Q)_r
},
\end{equation}
where
\begin{equation}\label{eq:C-ms}
C_{m,s}
=
\left(\frac Qa\right)^s
\frac{
(a;Q)_{2s}
}{
(1-a)
(aQ;Q)_{2s+N}
(Q^2;Q)_{2s}
}.
\end{equation}

Using
\begin{equation}\label{eq:qnegative-N}
\frac{1}{(Q;Q)_{N-r}}
=
\frac{(Q^{-N};Q)_r}{(Q;Q)_N}
(-1)^rQ^{Nr-\binom r2},
\end{equation}
which follows from
\[
(Q^{-N};Q)_r
=
(-1)^rQ^{-Nr+\binom r2}
\frac{(Q;Q)_N}{(Q;Q)_{N-r}},
\]

Substituting \eqref{eq:qnegative-N} into the sum in
\eqref{eq:L-factor}, we obtain
\begin{align}
&\frac1{(Q;Q)_N}
\sum_{r=0}^{N}
\frac{
(1-AQ^{2r})
(A,b,Q^{-N};Q)_r
}{
(Q,Q^{2s+2},AQ^{N+1};Q)_r
}
\nonumber\\
&\hspace{30mm}\times
(-1)^r
Q^{Nr-\binom r2}
\left(\frac Qa\right)^r.
\label{eq:pre-6phi5}
\end{align}

We apply Lemma~\ref{lem:degenerate-6phi5} with
\[
A=aQ^{2s},
\qquad
b=\frac aQ.
\]
Then
\[
\frac{AQ}{b}=Q^{2s+2}.
\]
Multiplying \eqref{eq:degenerate-6phi5} by \(1-A\) gives
\begin{align}
&\sum_{r=0}^{N}
\frac{
(1-AQ^{2r})(A,b,Q^{-N};Q)_r
}{
(Q,Q^{2s+2},AQ^{N+1};Q)_r
}
(-1)^rQ^{Nr-\binom r2}
\left(\frac Qa\right)^r
\nonumber\\
&\qquad
=
(1-A)
\left(\frac Qa\right)^N
\frac{(AQ;Q)_N}{(Q^{2s+2};Q)_N}.
\label{eq:6phi5-result}
\end{align}

Since
\[
A=aQ^{2s},
\]
we have
\[
(AQ;Q)_N
=
(aQ^{2s+1};Q)_N.
\]
Combining
\eqref{eq:L-factor},
\eqref{eq:C-ms}, and
\eqref{eq:6phi5-result} gives
\begin{align*}
L_{m,s}
={}&
\left(\frac Qa\right)^{s+N}
\frac{
(a;Q)_{2s}(1-aQ^{2s})
}{
(1-a)
}
\\
&\times
\frac{
(aQ^{2s+1};Q)_N
}{
(Q;Q)_N
(aQ;Q)_{2s+N}
(Q^2;Q)_{2s}
(Q^{2s+2};Q)_N
}.
\end{align*}

Now
\[
\frac{
(a;Q)_{2s}(1-aQ^{2s})
}{
1-a
}
=
(aQ;Q)_{2s},
\]
and
\[
(aQ;Q)_{2s}
(aQ^{2s+1};Q)_N
=
(aQ;Q)_{2s+N}.
\]
These factors cancel.

Similarly,
\[
(Q^2;Q)_{2s}
(Q^{2s+2};Q)_N
=
(Q^2;Q)_{2s+N}.
\]
Since
\[
s+N=m,
\qquad
2s+N=m+s,
\qquad
N=m-s,
\]
we obtain
\[
L_{m,s}
=
\left(\frac Qa\right)^m
\frac{1}{
(Q;Q)_{m-s}
(Q^2;Q)_{m+s}
}.
\]
This is precisely the right-hand side of
\eqref{eq:WP-entry}.
\end{proof}

\subsection{Reduction to an ordinary Bailey transition}
\label{sec:transition}

Recall
\[
B^{(p)}=W_{Q,a}A^{(p)}
\]
and
\[
\Gamma^{(p)}=DB^{(p)}.
\]
Theorem~\ref{thm:WP-diagonalization} gives
\begin{equation}\label{eq:Gamma-diagonalized}
M_a\Gamma^{(p)}
=
SM_QA^{(p)}.
\end{equation}
By Lemma~\ref{lem:Ap-Bailey},
\[
M_QA^{(p)}
=
\mathcal T_Q^p\delta.
\]
Thus
\begin{equation}\label{eq:Gamma-diagonalized2}
M_a\Gamma^{(p)}
=
S\mathcal T_Q^p\delta.
\end{equation}

On the other hand, by Lemma~\ref{lem:alpha-B},
\[
\Gamma_i^{(p)}
=
a^{pi}Q^{pi^2}
\alpha_i^{(n,p)}.
\]

Thus \(\Gamma^{(p)}\) is obtained from
\(\alpha^{(n,p)}\) by \(p\) successive limiting Bailey transformations
relative to \(a\).  
By Lemma~\ref{lem:limiting-bailey},
\[
M_a\Gamma^{(p)}
=
\mathcal T_a^p\beta^{(n,p)}.
\]

Recall that
\[
\beta^{(n,p)}
=
M_a\alpha^{(n,p)}.
\]

Combining this with \eqref{eq:Gamma-diagonalized2} yields
\begin{equation}\label{eq:T-a-p-beta}
\mathcal T_a^p\beta^{(n,p)}
=
S\mathcal T_Q^p\delta.
\end{equation}

The diagonal operator \(S\) has a particularly simple relation with
the Bailey operators.

\begin{lemma}\label{lem:TAS}
As operators on sequences,
\begin{equation}\label{eq:TAS}
\mathcal T_aS
=
\mathcal T_Q.
\end{equation}
Consequently,
\begin{equation}\label{eq:S-transition}
S
=
\mathcal T_a^{-1}\mathcal T_Q.
\end{equation}
\end{lemma}
\begin{proof}
For any sequence \(f\),
\begin{align*}
(\mathcal T_aSf)_m
&=
\sum_{j=0}^{m}
\frac{a^jQ^{j^2}}{(Q;Q)_{m-j}}
\left(\frac Qa\right)^j f_j \\
&=
\sum_{j=0}^{m}
\frac{Q^{j^2+j}}{(Q;Q)_{m-j}}f_j
=
(\mathcal T_Qf)_m.
\end{align*}
Thus \(\mathcal T_aS=\mathcal T_Q\).  Since
\(\mathcal T_a\) is invertible over \(\mathbb Q(Q,a)\),
\[
S=\mathcal T_a^{-1}\mathcal T_Q.
\]
\end{proof}

We have now obtained the main operator formula.

\begin{theorem}\label{thm:beta-transition}
For \(K_p=T(2,2p+1)\),
\begin{equation}\label{eq:beta-transition}
\beta^{(n,p)}
=
\mathcal T_a^{-(p+1)}
\mathcal T_Q^{p+1}\delta,
\qquad
a=Q^n.
\end{equation}
\end{theorem}

\begin{proof}
From \eqref{eq:T-a-p-beta},
\[
\beta^{(n,p)}
=
\mathcal T_a^{-p}
S\mathcal T_Q^p\delta.
\]
Using \eqref{eq:S-transition},
\begin{align*}
\beta^{(n,p)}
&=
\mathcal T_a^{-p}
\mathcal T_a^{-1}
\mathcal T_Q
\mathcal T_Q^p\delta
\\
&=
\mathcal T_a^{-(p+1)}
\mathcal T_Q^{p+1}\delta.
\end{align*}
\end{proof}

Thus Laurent integrality reduces to the integrality of the transition
operator
\[
\mathcal T_{Q^n}^{-r}\mathcal T_Q^r,
\qquad r\geq0.
\]

\section{Integral transitions between Bailey operators}
\label{sec:integrality}

We work with an independent Bailey parameter \(c\).  Set
\[
\mathscr R=\Z[Q^{\pm1},c^{\pm1}],
\qquad
\mathscr F=\Q(Q,c),
\]
where \(c\) is an indeterminate.
All formal power-series rings in this section are endowed with the \(z\)-adic topology.  
We identify a sequence \(f=(f_0,f_1,\ldots)\) with its generating series
\[
F_f(z)=\sum_{m\geq0}f_mz^m.
\]

\subsection{Factorization of a Bailey operator}

Put
\begin{equation}\label{eq:Phi-def}
\Phi(z)
=
\sum_{r\geq0}\frac{z^r}{(Q;Q)_r}
\in\Q(Q)[[z]].
\end{equation}
Since \(\Phi(0)=1\), multiplication by \(\Phi(z)\) is an automorphism
of \(\mathscr F[[z]]\).  Define
\[
\mathcal C F(z)=\Phi(z)F(z)
\]
and define the continuous diagonal automorphism \(\mathcal E_c\) by
\begin{equation}\label{eq:Ec-operator}
\mathcal E_c(z^m)=c^mQ^{m^2}z^m.
\end{equation}
The coefficient identity
\begin{equation}\label{eq:Phi-functional}
(1-z)\Phi(z)=\Phi(Qz)
\end{equation}
follows directly from
\((Q;Q)_r=(1-Q^r)(Q;Q)_{r-1}\).

\begin{lemma}\label{lem:T-factorization-general}
As continuous operators on \(\mathscr F[[z]]\),
\begin{equation}\label{eq:T-factorization-general}
\mathcal T_c=\mathcal C\mathcal E_c.
\end{equation}
\end{lemma}

\begin{proof}
If \(F(z)=\sum_{j\geq0}f_jz^j\), then the coefficient of \(z^m\) in
\(\mathcal C\mathcal E_cF\) is
\[
\sum_{j=0}^{m}
\frac{c^jQ^{j^2}}{(Q;Q)_{m-j}}f_j,
\]
which is exactly \((\mathcal T_cf)_m\).
\end{proof}
Since both \(\mathcal C\) and \(\mathcal E_c\) are continuous
automorphisms of \(\mathscr F[[z]]\), so is \(\mathcal T_c\).

For \(\ell\in\Z\), let
\[
\tau_\ell F(z)=F(Q^\ell z).
\]

\begin{lemma}\label{lem:T-shift-general}
For every \(\ell\in\Z\),
\begin{equation}\label{eq:T-shift-general}
\mathcal T_{cQ^\ell}=\mathcal T_c\tau_\ell.
\end{equation}
\end{lemma}

\begin{proof}
On \(z^m\),
\[
\mathcal E_c\tau_\ell(z^m)
=
c^mQ^{m^2+\ell m}z^m
=
\mathcal E_{cQ^\ell}(z^m).
\]
Thus \(\mathcal E_{cQ^\ell}=\mathcal E_c\tau_\ell\), and
\eqref{eq:T-shift-general} follows from
Lemma~\ref{lem:T-factorization-general}.
\end{proof}

\subsection{An integral \(q\)-difference algebra}
Let \(\mathscr D_{\mathscr R}\) consist of all formal operators
\begin{equation}\label{eq:D-def-general}
\mathcal O
=
\sum_{s\geq0}\sum_{k\in\mathbb Z}
a_{s,k}z^s\tau_k,
\qquad
a_{s,k}\in\mathscr R,
\end{equation}
such that, for each fixed \(s\), only finitely many \(a_{s,k}\) are nonzero.
These operators act continuously on \(\mathscr R[[z]]\).
Indeed,
\[
z^s\tau_k(z^m)=Q^{km}z^{m+s},
\]
and the coefficient of a fixed power \(z^M\) receives contributions
only from \(s\leq M\).  Moreover,
\begin{equation}\label{eq:D-composition-general}
(z^s\tau_k)(z^t\tau_h)
=
Q^{kt}z^{s+t}\tau_{k+h}.
\end{equation}
For a fixed total power \(z^M\), only finitely many pairs \((s,t)\)
with \(s+t=M\) occur, and for each such pair only finitely many
dilation indices occur.  Hence \(\mathscr D_{\mathscr R}\) is closed
under composition.

\begin{lemma}\label{lem:C-conjugation-general}
For \(s\geq0\) and \(k\in\Z\),
\begin{equation}\label{eq:C-conjugation-general}
\mathcal C^{-1}z^s\tau_k\mathcal C
=
z^sP_k(z)\tau_k,
\qquad
P_k(z)=\frac{\Phi(Q^kz)}{\Phi(z)}.
\end{equation}
Moreover \(P_k(z)\in\Z[Q^{\pm1}][[z]]\), and hence
\begin{equation}\label{eq:C-preserves-D-general}
\mathcal C^{-1}\mathscr D_{\mathscr R}\mathcal C
\subseteq
\mathscr D_{\mathscr R}.
\end{equation}
\end{lemma}

\begin{proof}
The conjugation formula follows directly from the definition of
\(\mathcal C\).  Iterating \eqref{eq:Phi-functional}, for \(k\geq0\)
we obtain
\[
P_k(z)=(z;Q)_k\in\Z[Q][z].
\]
If \(k=-d<0\), then
\[
P_{-d}(z)=\frac1{(Q^{-d}z;Q)_d}
=
\sum_{r=0}^{\infty}
Q^{-dr}\Gauss{d+r-1}{r}z^r
\in\Z[Q^{\pm1}][[z]],
\]
where the second equality is \eqref{eq:negative-q-binomial}.  Expanding
\(P_k(z)\) preserves the local-finiteness condition in
\eqref{eq:D-def-general}, which proves
\eqref{eq:C-preserves-D-general}.
\end{proof}

\begin{lemma}\label{lem:Ec-conjugation}
For \(s\geq0\) and \(k\in\Z\),
\begin{equation}\label{eq:Ec-conjugation}
\mathcal E_c^{-1}z^s\tau_k\mathcal E_c
=
c^{-s}Q^{-s^2}z^s\tau_{k-2s}.
\end{equation}
Consequently,
\begin{equation}\label{eq:Ec-preserves-D}
\mathcal E_c^{-1}\mathscr D_{\mathscr R}\mathcal E_c
\subseteq
\mathscr D_{\mathscr R}.
\end{equation}
\end{lemma}

\begin{proof}
Evaluating the left-hand side of \eqref{eq:Ec-conjugation} on \(z^m\)
gives
\begin{align*}
\mathcal E_c^{-1}z^s\tau_k\mathcal E_c(z^m)
&=
c^mQ^{m^2+km}
\mathcal E_c^{-1}(z^{m+s})
\\
&=
c^{-s}Q^{m^2+km-(m+s)^2}z^{m+s}
\\
&=
c^{-s}Q^{-s^2}Q^{(k-2s)m}z^{m+s},
\end{align*}
which is the right-hand side of \eqref{eq:Ec-conjugation}.  Since
\(c^{-s}Q^{-s^2}\in\mathscr R\), conjugation sends each monomial
generator of \(\mathscr D_{\mathscr R}\) back into
\(\mathscr D_{\mathscr R}\), and the local-finiteness condition is
unchanged.
\end{proof}

\subsection{The uniform integral transition theorem}

\begin{theorem}[Integral Bailey transition]
\label{thm:general-integral-transition}
Let \(c\) be an indeterminate, \(\ell\in\Z\), and \(r\geq0\).  Then
\[
\mathcal U_{c,\ell}^{(r)}
:=
\mathcal T_c^{-r}\mathcal T_{cQ^\ell}^{r}
\]
defines a continuous \(\mathscr R\)-linear automorphism of
\(\mathscr R[[z]]\).  Consequently, all of its matrix coefficients in
the monomial basis belong to
\[
\mathscr R=\Z[Q^{\pm1},c^{\pm1}].
\]
Its diagonal entry in degree \(m\) is \(Q^{\ell rm}\).
\end{theorem}

\begin{proof}
The case \(r=0\) is immediate.  Assume \(r\geq1\), set
\[
\mathcal A=\mathcal T_c,
\qquad
\tau=\tau_\ell,
\]
and use Lemma~\ref{lem:T-shift-general} to write
\[
\mathcal T_{cQ^\ell}=\mathcal A\tau.
\]
Hence
\begin{equation}\label{eq:general-transition-factorization}
\mathcal U_{c,\ell}^{(r)}
=
\mathcal A^{-r}(\mathcal A\tau)^r.
\end{equation}
As a product of conjugates,
\begin{equation}\label{eq:general-conjugate-product}
\mathcal A^{-r}(\mathcal A\tau)^r
=
\left(\mathcal A^{-(r-1)}\tau\mathcal A^{r-1}\right)
\cdots
\left(\mathcal A^{-1}\tau\mathcal A\right)\tau.
\end{equation}

By Lemma~\ref{lem:T-factorization-general},
\(\mathcal A=\mathcal C\mathcal E_c\).  Hence
\[
\mathcal A^{-1}\mathscr D_{\mathscr R}\mathcal A
=
\mathcal E_c^{-1}
\mathcal C^{-1}\mathscr D_{\mathscr R}\mathcal C
\mathcal E_c
\subseteq
\mathcal E_c^{-1}\mathscr D_{\mathscr R}\mathcal E_c
\subseteq
\mathscr D_{\mathscr R}.
\]

Since \(\tau\in\mathscr D_{\mathscr R}\), induction shows that every
conjugate
\[
\mathcal A^{-j}\tau\mathcal A^j
\]
belongs to \(\mathscr D_{\mathscr R}\).  The product
\eqref{eq:general-conjugate-product} therefore lies in
\(\mathscr D_{\mathscr R}\), and consequently preserves
\(\mathscr R[[z]]\).  This proves integrality of all matrix
coefficients.

Since
\[
\Z[Q^{\pm1},(cQ^\ell)^{\pm1}]=\mathscr R,
\]
the same argument applied to the pair \((cQ^\ell,-\ell)\) shows that
\[
\mathcal T_{cQ^\ell}^{-r}\mathcal T_c^r
\]
also preserves \(\mathscr R[[z]]\).  This operator is the inverse of
\(\mathcal U_{c,\ell}^{(r)}\), so the latter is an automorphism.
Finally, the diagonal entry of \(\mathcal T_c\) in degree \(m\) is
\(c^mQ^{m^2}\), whereas that of \(\mathcal T_{cQ^\ell}\) is
\(c^mQ^{m^2+\ell m}\).  Taking the indicated powers gives the diagonal
entry \(Q^{\ell rm}\) of \(\mathcal U_{c,\ell}^{(r)}\).
\end{proof}

\begin{corollary}\label{cor:integral-transition}
For every integer \(n\) and every \(r\geq0\),
\begin{equation}\label{eq:integral-transition}
\mathcal T_{Q^n}^{-r}\mathcal T_Q^r
\end{equation}
has all matrix coefficients in \(\Z[Q^{\pm1}]\) and is an automorphism
of \(\Z[Q^{\pm1}][[z]]\).
\end{corollary}

\begin{proof}
In Theorem~\ref{thm:general-integral-transition}, specialize
\[
c=Q^n,
\qquad
\ell=1-n.
\]
Then \(cQ^\ell=Q\).  Since every matrix coefficient in the general
theorem is a Laurent polynomial in \(Q\) and \(c\), this specialization
is well-defined and belongs to \(\Z[Q^{\pm1}]\).  The same
specialization applies to the inverse.
\end{proof}

\subsection{Proof of Theorem~\ref{thm:main}}

\begin{proof}[Proof of Theorem~\ref{thm:main}]
By Proposition~\ref{prop:H-universal}, there is a unique sequence
\[
H_m^{(n,p)}\in\Q(q)
\]
for which \eqref{eq:main-expansion} holds.  Define
\[
\beta_m^{(n,p)}
=
(-1)^m q^{-m(m+n+1)}H_m^{(n,p)}.
\]
Theorem~\ref{thm:beta-transition} gives
\[
\beta^{(n,p)}
=
\mathcal T_{Q^n}^{-(p+1)}
\mathcal T_Q^{p+1}\delta.
\]
Together with the definition of \(\beta_m^{(n,p)}\), this gives
\eqref{eq:main-H-operator}.
By Corollary~\ref{cor:integral-transition}, applied with \(r=p+1\),
\[
\beta_m^{(n,p)}\in\Z[Q^{\pm1}]
\qquad(m\geq0).
\]
Therefore
\[
H_m^{(n,p)}
=
(-1)^m q^{m(m+n+1)}\beta_m^{(n,p)}
\in\Z[q^{\pm1}],
\]
because \(\Z[Q^{\pm1}]=\Z[q^{\pm2}]\subset\Z[q^{\pm1}]\).
The mirror statement follows from Corollary~\ref{cor:mirror}.
\end{proof}

\begin{proof}[Proof of Corollary~\ref{cor:all-two-strand}]
If \(k>0\) is odd, write \(k=2p+1\) with \(p\geq0\) and apply
Theorem~\ref{thm:main}.  If \(k<0\), then \(T(2,k)\) is the mirror of
\(T(2,-k)\), so the conclusion follows from the mirror part of the
theorem. 
\end{proof}

\section{Further discussions}
\label{sec:related}

We conclude with several remarks on possible extensions of the present
method and on its relation to congruence properties and the
\(SU(n)\) volume conjecture.

\subsection{Higher-strand torus knots}

The Newton expansion in Section~\ref{sec:QDD} is valid for arbitrary
knots; the restriction to \(T(2,2p+1)\) enters only through the
reduction of the Lin--Zheng formula to a one-dimensional WP-Bailey
transform.  For a general torus knot \(T(r,k)\), with \(\gcd(r,k)=1\),
the Lin--Zheng formula involves the plethysm
\[
s_{(N)}(x_1^r,x_2^r,\ldots)
=
\sum_{\lambda\vdash rN}
c_{(N);r}^{\lambda}s_\lambda(x_1,x_2,\ldots);
\]
see \cite{LinZheng}.  For \(r>2\), this leads to multi-index sums,
suggesting a higher-dimensional analogue of the Bailey and WP-Bailey
transforms.

The integral transition theorem of Section~\ref{sec:integrality}, for \(d\geq0\),
\[
\mathcal T_c^{-d}\mathcal T_{cQ^\ell}^{\,d}
\in
\operatorname{Mat}_{\mathrm{lt}}
\bigl(\Z[Q^{\pm1},c^{\pm1}]\bigr),
\]
is independent of the torus-knot calculation.  Thus, if the
higher-strand coefficients admit a suitable Bailey-type reduction,
the integrality argument may extend to this setting.  A natural
problem is to find a higher-dimensional analogue of the
diagonalization
\[
M_aD W_{Q,a}=S M_Q;
\]
see \cite{AndrewsBerkovich,McLaughlin,Warnaar}.

\subsection{Bailey structures for arbitrary knots}

Theorem~\ref{thm:universal-bailey} associates to every knot \(K\) an
ordinary Bailey pair
\[
\beta^{(n)}(K)=M_a\alpha^{(n)}(K),
\qquad a=Q^n,
\]
with \(\alpha^{(n)}(K)\) determined by its colored \(SU(n)\) invariants.
Thus the Newton transform has a formal Bailey description for arbitrary
knots.

For \(T(2,2p+1)\), this formal relation admits the explicit factorization
\[
\beta^{(n,p)}
=
\mathcal T_{Q^n}^{-(p+1)}
\mathcal T_Q^{p+1}\delta.
\]
It is natural to ask which other knot families admit comparably explicit
Bailey or WP-Bailey descriptions.

\subsection{Congruence relations}

Congruence relations for colored HOMFLY-PT and \(SU(n)\) invariants
were one of the motivations for the Chen--Liu--Zhu cyclotomic
expansion; see \cite{CLPZ,CLZ,ZhuStructures}.  In the present
framework, the corresponding divisibility follows directly from the
Newton expansion.

\begin{corollary}\label{cor:further-congruence}
Let \(K\) be a knot such that
\[
J_N^{SU(n)}(K;q)
=
\sum_{m=0}^{N}
P_m(X_N)H_m^{(n)}(K;q),
\qquad
H_m^{(n)}(K;q)\in\Z[q^{\pm1}],
\]
where
\[
P_m(X)=\prod_{j=0}^{m-1}(X-X_j).
\]
Then, for \(N\geq k\),
\begin{equation}\label{eq:further-congruence}
J_N^{SU(n)}(K;q)
\equiv
J_k^{SU(n)}(K;q)
\pmod{\{N-k\}\{N+k+n\}}.
\end{equation}
\end{corollary}

\begin{proof}
Since \(P_m(X_k)=0\) for \(m>k\),
\[
J_N^{SU(n)}(K;q)-J_k^{SU(n)}(K;q)
=
\sum_{m=0}^{N}
\bigl(P_m(X_N)-P_m(X_k)\bigr)
H_m^{(n)}(K;q).
\]
For every \(m\), the polynomial
\(P_m(X_N)-P_m(X_k)\) is divisible by \(X_N-X_k\).  Since
\[
X_N-X_k
=
\{N-k\}\{N+k+n\},
\]
the result follows.
\end{proof}

Thus Theorem~\ref{thm:main} proves the congruence \eqref{eq:further-congruence} for all two-strand torus knots considered in this paper.

\subsection{The \(SU(n)\) volume conjecture}

The volume conjecture of Kashaev and Murakami--Murakami relates the
asymptotics of the colored Jones polynomial at roots of unity to the hyperbolic geometry of knot complements
\cite{Kashaev,MurakamiMurakami}.  Related volume conjectures have also
been formulated for the Reshetikhin--Turaev and Turaev--Viro invariants
\cite{ChenYang}.

Chen--Liu--Zhu formulated an \(SU(n)\) analogue of the volume
conjecture \cite{CLZ}.  For fixed integers \(a,s\), set
\[
\xi_{N,a}(s)
=
\exp\left(
\frac{s\pi\sqrt{-1}}{N+a}
\right).
\]
Their conjecture has two regimes.

\begin{conjecture}[Chen--Liu--Zhu]\label{conj:SU-n-volume}
Let \(n\geq2\).

\begin{enumerate}
\item[(i)]
If
\[
a\in\mathbb Z\setminus\{1,\ldots,n-1\},
\]
then, for every knot \(K\),
\[
2\pi s
\lim_{N\to\infty}
\frac{
\log J_N^{SU(n)}
\bigl(K;\xi_{N,a}(s)\bigr)
}{
N+1
}
=
0.
\]

\item[(ii)]
If
\[
a\in\{1,\ldots,n-1\},
\]
then, for every hyperbolic knot \(K\),
\[
2\pi s
\lim_{N\to\infty}
\frac{
\log J_N^{SU(n)}
\bigl(K;\xi_{N,a}(s)\bigr)
}{
N+1
}
=
\operatorname{Vol}(S^3\setminus K)
+
\sqrt{-1}\,
\operatorname{CS}(S^3\setminus K).
\]
\end{enumerate}
\end{conjecture}

\begin{corollary}\label{cor:zero-volume-regime}
Let \(K_p=T(2,2p+1)\), \(n\geq2\), and
\[
a\in\mathbb Z\setminus\{1,\ldots,n-1\}.
\]
Then
\[
2\pi s
\lim_{N\to\infty}
\frac{
\log J_N^{SU(n)}
\bigl(K_p;\xi_{N,a}(s)\bigr)
}{
N+1
}
=
0.
\]
Hence part {\rm (i)} of
Conjecture~\ref{conj:SU-n-volume} holds for all two-strand torus knots.
\end{corollary}

\begin{proof}
Put
\[
q_N=\xi_{N,a}(s)
=
\exp\left(\frac{s\pi\sqrt{-1}}{N+a}\right).
\]
Then
\[
\{N+a\}_{q_N}=0.
\]

If \(a\leq0\), the factor
\[
\{N-(-a)\}=\{N+a\}
\]
occurs in \(C_{N+1,m}^{(n)}\) whenever \(m>-a\).  If \(a\geq n\),
the factor
\[
\{N+n+(a-n)\}=\{N+a\}
\]
occurs whenever \(m>a-n\).  Hence, for all sufficiently large \(N\),
the cyclotomic expansion contains only finitely many nonzero terms.

For every fixed integer \(b\),
\[
\{N+b\}_{q_N}
=
2\sqrt{-1}(-1)^s
\sin\left(
\frac{s\pi(b-a)}{N+a}
\right)
=
O(N^{-1}).
\]
Consequently, for every fixed \(m\geq1\),
\[
C_{N+1,m}^{(n)}(q_N)
=
O(N^{-2m}).
\]
Since
\[
H_m^{(n)}(K_p;q)\in\mathbb Z[q^{\pm1}],
\]
we have
\[
H_m^{(n)}(K_p;q_N)
\longrightarrow
H_m^{(n)}(K_p;1),
\]
and in particular \(H_m^{(n)}(K_p;q_N)\) remains bounded as
\(N\to\infty\).  Since \(H_0^{(n)}=J_0^{SU(n)}=1\), it follows that
\[
J_N^{SU(n)}(K_p;q_N)\longrightarrow1.
\]

Thus, for all sufficiently large \(N\) these values lie in the disk
\(D(1,1/2)\), on which the principal branch of the logarithm is
holomorphic.  Hence
\[
\log J_N^{SU(n)}(K_p;q_N)\longrightarrow 0.
\]
Therefore,
\[
2\pi s
\lim_{N\to\infty}
\frac{\log J_N^{SU(n)}(K_p;q_N)}{N+1}
=0.
\]
\end{proof}

\section*{Acknowledgements}

The author is especially grateful to Professor Qingtao Chen. During a
visit to Shanghai, Professor Chen suggested the problem of the
cyclotomic expansion conjecture for colored \(SU(n)\) invariants. That
conversation provided the initial motivation for the present work.
The author is thanks him for foundational work on the theory
of colored \(SU(n)\) invariants, for his generosity in sharing ideas,
and for many illuminating discussions. His enthusiasm for the subject
has also been a constant source of inspiration to the author.

The author also wishes to express his heartfelt gratitude to Professor
Huabin Ge for his constant support and guidance. Professor Ge took part
in many discussions related to this problem and, during the author's
time in Shanghai, introduced the author to Professor Chen, making these
valuable exchanges possible. The author is deeply thankful to him for
the trust and support he has given throughout the author's research.

\end{document}